\documentclass[12pt]{article}
\usepackage{amsmath, amsthm}
\usepackage{enumerate}
\usepackage{amssymb}
\usepackage{geometry}
\newtheorem{thm}{Theorem}[section]

\newtheorem{ques}[thm]{Question}
\newtheorem{prop}[thm]{Proposition}

\theoremstyle{definition}
\newtheorem{defn}{Definition}[section]
\theoremstyle{remark}
\newtheorem{rem}{Remark}[section]

\newcommand{\eval}[2][\right]{\relax
  \ifx#1\right\relax \left.\fi#2#1\rvert}

\begin{document}
\title{\bf Topological algebras of bounded operators with
locally solid Riesz spaces} 
\maketitle

\author{\centering ABDULLAH AYDIN \\ \bigskip  \small  
	Department of Mathematics, Mu\c{s} Alparslan University, Mu\c{s}, Turkey. \\}

\bigskip

\abstract{Let $X$ be a vector lattice and $(E,\tau)$ be a locally solid vector lattice. An operator $T:X\to E$ is said to be $ob$-bounded if, for each order bounded set $B$ in $X$, $T(B)$ is topologically bounded in $E$. In this paper, we study on algebraic properties of $ob$-bounded operators with respect to the topology of uniform convergence and equicontinuous convergence.

\bigskip
\let\thefootnote\relax\footnotetext
{
	
\text{2010 AMS Mathematics Subject Classification:} 46A40, 47B65, 46H35

e-mail: a.aydin@alparslan.edu.tr}

\section{Introduction and preliminaries}\label{A1}
Bounded operators play an important role in the functional analysis. Our aim in this paper is to introduce and study $ob$-bounded operators from vector lattice to locally solid vector lattices or between locally solid vector lattices, which attracted the attention of several authors in series of recent papers; for example \cite{AA,AEEM2,Tr,Z}. By an {\em operator}, we always mean a linear operator between vector spaces.

Let us recall some notation and terminology used in this paper. In every topological vector space (over $\mathbb{R}$ or $\mathbb{C}$), there exists a base $N_0$ of zero neighborhoods with; every $V\in N_0$, $\lambda V\subseteq V$ whenever $\lvert \lambda\rvert\leq 1$; for every $V_1,V_2\in N_0$ there exists $V\in N_0$ such that $V\subseteq V_1\cap V_2$; for every $V\in N_0$ there exists $U\in N_0$ such that $U+U\subseteq V$; for every $V\in N_0$ and every scalar $\lambda$ the set $\lambda V$ is in $N_0$. Whenever we mention a neighborhood of zero, we always assume that it belongs to
a base which satisfies these properties.  A subset $A$ of a vector lattice $E$ is said to be {\em solid} if, for $y\in A$ and $x\in E$ with $\lvert x\rvert\leq\lvert y\rvert$, we have $x\in A$. Let $E$ be a vector lattice and $\tau$ be a linear topology on $E$ that has a base at zero consisting of solid sets. Then the pair $(E,\tau)$ is said a {\em locally solid vector lattice} (or, {\em locally solid Riesz space}) that means a topological vector lattice with a locally solid topology, for more details on these notions, see \cite{AB,ABPO}. 

A vector lattice $E$ is {\em order complete} if every subset of $E$ which is bounded above has a supremum. For $a\in E_+$, the {\em order interval} consists of all $x\in E$ such that $-a\leq x\leq a$. In a vector lattice $E$, a subset $B\subseteq E$ is called {\em order bounded} if it is contained in an order interval. Similarly, in a topological vector lattice $(E,\tau)$, a subset $B\subseteq E$ is called {\em bounded} (or, {\em topological bounded}) if, for each zero neighborhood $U\in E$, there exists a positive scalar $\lambda$ with $B\subseteq \lambda U$. We refer the reader for more information about some relation and result for mentioned types of bounded sets to \cite{AB,ABPO,L,Tr}. 

An {\em order bounded operator} between vector lattices sends order bounded subsets to order bounded subsets. For given a vector lattices $E$, $B_b(E)$ is the space of all order bounded operators on $E$. Since we have just order structure on $B_b(E)$, there is not exist a topological structure on it. But, for given a Banach lattice $E$, $B_b(E)$ forms a Banach lattice. Through this paper, we consider a subspace of the space of all linear operators between vector lattices. All vector spaces in this paper are assumed to be real, and all topological vector lattices are considered to be locally solid.

Now, let us give some known types of bounded operators. Let $E$ and $F$ be topological vector spaces. An operator $T:E\to F$ is said to be {\em $bb$-bounded} if it maps every bounded set into a bounded set, and said to be {\em $nb$-bounded} if $T$ maps some neighborhood of zero into a bounded set; see \cite{Tr}. When $E$ and $F$ are locally solid vector lattices, $T$ is said to be {\em $bo$-bounded} if, for every bounded set $B\subset E$, $T(B)$ is order bounded in $F$, and said to be {\em $no$-bounded} if there exists some zero neighborhood $U\subset E$ such that $T(U)$ is order bounded in $F$; see \cite{Z}. Motivated by these definitions, we give the following notion.

\begin{defn}
Let $X$ be a vector lattice and $(E,\tau)$ be a locally solid vector lattice. Then an operator $T:X\to E$ is said to be {\em $ob$-bounded} if, for every order bounded set $B\subseteq X$, $T(B)$ is $\tau$-bounded in $E$.
\end{defn}

Every order bounded subset is topologically bounded in locally solid vector lattices; see \cite[Thm.2.19(i)]{AB}, and so we have the following result.
\begin{rem}\
\begin{enumerate}
\item[(i)] Every $bb$-bounded operator between locally solid vector lattices is $ob$-bounded.
\item[(ii)] Every $bo$-bounded operator is $ob$-bounded.
\item[(iii)] Every order bounded operator from a vector lattice to a locally solid vector lattice is $ob$-bounded.
\end{enumerate}
\end{rem} 

It is natural to ask whether an $ob$-bounded operator is order bounded or $bo$-bounded. By considering the Example 2.4 of \cite{L}, we have the negative answer for order boundedness since bounded set may be not order bounded in locally solid vector lattice. Also, by considering the Example 2 of \cite{Z} and the Theorem \cite[Thm.2.19(i)]{AB}, the identity operator $I$ in the example is $ob$-bounded but it fails to be $bo$-bounded. But, on the other hand, we have a partial positive answer. Every bounded set is order bounded in locally solid vector lattice when space contains an order bounded zero neighborhood; see \cite[Thm.2.2]{L}. So, every $ob$-bounded operator from a vector lattice to a locally solid vector lattice with an order bounded zero neighborhood is order bounded, and also every $ob$-bounded operator between locally solid vector lattices with order bounded zero neighborhoods is $bo$-bounded, and lastly every $ob$-bounded operator from a locally solid vector lattice with an order bounded zero neighborhood to a locally solid vector lattice is $bb$-bounded. 

\begin{rem}\label{some example of bounded operators}
The sum of two $ob$-bounded operators is $ob$-bounded since the sum of two bounded sets in a topological vector space is bounded. The product of two $ob$-bounded operators may not be $ob$-bounded. However, the product of an $ob$-bounded operator $T$ with a $bo$-bounded operator $S$ from the left hand (i.e $S\circ T$) is order bounded and so $ob$-bounded.
\end{rem}

The class of all $ob$-bounded operators from a vector lattice $X$ to a locally solid vector lattice $(E,\tau)$ is denoted by $B_{ob}(X,E)$, and it will be equipped with the topology of uniform convergence on order bounded sets. Recall that a net $(S_\alpha)$ of $ob$-bounded operators is said to {\em converge to zero uniformly} on an order bounded set $B$ if, for each zero neighborhood $V$ in $E$, there exists an index $\alpha_0$ such that $S_\alpha(B)\subseteq V$ for all $\alpha\geq\alpha_0$. We say that $(S_\alpha)$ converges to
$S$ uniformly on order bounded sets if $(S_\alpha-S)$ converges to zero uniformly on order bounded sets; see \cite[2.16]{Tr}. Also, $ob$-bounded operators will be equipped with the topology of equicontinuous convergence with order bounded sets. Recall that a net $(S_\alpha)$  of $ob$-bounded operators is said to {\em converge to zero equicontinuously} with order bounded sets if, for each zero neighborhood $V$ in $Y$, there is an order bounded set $B$ in $X$ such that, for every $\varepsilon>0$, there exists an index $\alpha_0$ such that $S_\alpha(B) \subseteq\varepsilon V$ for all $\alpha\geq\alpha_0$. We say that $(S_\alpha)$ converges to $S$ equicontinuously with order bounded sets if $(S_\alpha-S)$ converges to zero equicontinuously with order bounded sets; see \cite[2.18]{Tr}.
\section{Main results}\label{A2}

In the following two results, we show the continuity of addition and scalar multiplication with respect to the uniform convergence topology and the equicontinuous convergence topology, respectively.
\begin{thm}
The operations of addition and scalar multiplication are continuous in $B_{ob}(X,E)$ with respect to the uniform convergence topology on order bounded sets.
\end{thm}

\begin{proof}
Suppose $(T_\alpha)$ and $(S_\alpha)$ are two nets of $ob$-bounded operators which are uniform convergent to zero on order bounded sets. Fix an arbitrary order bounded set $B$ in $X$. Consider a zero neighborhood $V$ in $E$ then there exists another zero neighborhood $U$ such that $U+U\subseteq V$. So that there are some indexes $\alpha_1$ and $\alpha_2$ such that $T_\alpha(B)\subseteq U$ for every $\alpha\geq\alpha_1$ and $S_\alpha(B)\subseteq U$ for each $\alpha\geq\alpha_2$. Since index set is directed, there is another index $\alpha_0$ such that $\alpha_0\geq\alpha_1$ and $\alpha_0\geq\alpha_1$. Hence, $T_\alpha(B)\subseteq U$ and $S_\alpha(B)\subseteq U$ for all $\alpha\geq\alpha_0$. Then we have
$$
(T_\alpha+S_\alpha)(B)\subseteq T_\alpha(B)+S_\alpha(B)\subseteq  U+U\subseteq V
$$
for each $\alpha\geq\alpha_0$. Therefore, since $B$ is arbitrary, we get that addition is continuous in $B_{ob}(E)$ with respect to the uniform convergence topology on order bounded sets.

Next, we show the continuity of the scalar multiplication. Consider the order bounded set $B$ in $E$ and a sequence of reals $\lambda_n$ which is convergent to zero. Since $(T_\alpha)$ is uniform convergent to zero on order bounded sets, for each zero neighborhood $V$ in $E$, there exists an index $\alpha_0$ such that $T_\alpha(B)\subseteq V$ for all $\alpha\geq\alpha_0$. For sufficiently large $n$, we have $\lvert\lambda_n\rvert\leq 1$, and so $\lambda_n V\subseteq V$. Then, for all $\alpha\geq\alpha_0$ and sufficiently large $n$, we have
$$
\lambda_nT(B)=T(\lambda_n B)\subseteq\lambda_n V \subseteq V.
$$
Therefore, we get the desired result.
\end{proof}

\begin{thm}
The operations of addition and scalar multiplication are continuous in $B_{ob}(X,E)$ with respect to the equicontinuous convergence topology with order bounded sets.
\end{thm}

\begin{proof}
Let $(T_\alpha)$ and $(S_\alpha)$ be two nets of $ob$-bounded operators which are equicontinuous convergent to zero with order bounded sets. Fix an arbitrary zero neighborhood $V$. Then there is another zero neighborhood $U$ with $U+U\subseteq V$. So, there exist order bounded sets $B_1$, $B_2$ in $X$ such that, for every $\varepsilon>0$, there are indexes $\alpha_1$ and $\alpha_1$ such that $T_\alpha(B_1) \subseteq\varepsilon U$ for all $\alpha\geq\alpha_1$ and $S_\alpha(B_2) \subseteq\varepsilon U$ for each $\alpha\geq\alpha_2$. Take an index $\alpha_0$ such that $\alpha_0\geq\alpha_1$ and $\alpha_0\geq\alpha_1$. Hence, $T_\alpha(B_1)\subseteq\varepsilon U$ and $S_\alpha(B_2) \subseteq\varepsilon U$ for all $\alpha\geq\alpha_0$. Choose the order bounded set $B=B_1\cap B_2$ then we have 
$$
(T_\alpha+S_\alpha)(B)\subseteq T_\alpha(B_1)+S_\alpha(B_2)\subseteq=\varepsilon U+\varepsilon U\subseteq \varepsilon V
$$
for each $\alpha\geq\alpha_0$. Therefore, since $V$ is arbitrary, we get the desired result.

Now, we show the continuity of scalar multiplication. Consider the zero neighborhood $V$ and a sequence of reals $\lambda_n$ which is convergent to zero. Since $(T_\alpha)$ is equicontinuous convergent to zero with order bounded sets, there exists an order bounded set $B$ in $X$ such that, for every $\varepsilon>0$, there is an index $\alpha_0$ such that $T_\alpha(B)\subseteq\varepsilon V$ for each $\alpha\geq\alpha_0$. Also, for sufficiently large $n$, we have $\lvert\lambda_n\rvert\leq 1$, so that $\lambda_n V\subseteq V$. Then, for all $\alpha\geq\alpha_0$ and for fixed sufficiently large $n$, we have
$$
\lambda_nT_\alpha(B)=T_\alpha(\lambda_n B)\subseteq \lambda_n\varepsilon V \subseteq\varepsilon V.
$$
Therefore, we get the desired result.
\end{proof}

The next two results give the continuity of the product of $ob$-bounded operators  with respect to the uniform convergence topology and the equicontinuous convergence topology, respectively.
\begin{thm}
Let $(E,\tau)$ be locally solid vector lattice with an order bounded zero neighborhood. Then the product of $ob$-bounded operators is continuous in $B_{ob}(E)$ with respect to the uniform convergence topology on order bounded sets.
\end{thm}

\begin{proof}
Suppose $(T_\alpha)$ and $(S_\alpha)$ which are uniform convergent to zero on order bounded sets are two nets of $ob$-bounded operators. Fix a zero neighborhood $V$ in $E$, so there is $a\in E_+$ such that $V\subseteq [-a,a]$; see \cite[Thm.2.2]{L}. Since $(T_\alpha)$ is uniform convergent to zero, there exists an index $\alpha_1$ such that $T_\alpha([-a,a])\subseteq V$ for all $\alpha\geq\alpha_1$. Also, there is an index $\alpha_2$ such that  $S_\alpha([-a,a])\subseteq V$ for all $\alpha\geq\alpha_2$. Since there is an index $\alpha_0$ with $\alpha_0\geq\alpha_1$ and $\alpha_0\geq\alpha_2$, we have
$$
S_\alpha\big(T_\alpha([-a,a])\big)\subseteq S_\alpha(V)\subseteq  S_\alpha([-a,a])\subseteq V
$$
for each $\alpha\geq\alpha_0$. Therefore, we get the desired result.
\end{proof}

\begin{thm}
Let $(E,\tau)$ be locally solid vector lattice with an order bounded zero neighborhood. Then the product of $ob$-bounded operators is continuous in $B_{ob}(E)$ with respect to the equicontinuous convergence topology with order bounded sets.
\end{thm}

\begin{proof}
Let $(T_\alpha)$ and $(S_\alpha)$ be nets of $ob$-bounded operators which are equicontinuous convergent to zero with order bounded sets. Fix a zero neighborhood $V$ in $E$ and, by applying \cite[Thm.2.2]{L}, there is $a\in E_+$ such that $V\subseteq [-a,a]$. On the other hand, there exist order bounded sets $B_1$ and $B_2$ in $X$ such that, for every $\varepsilon>0$, there are indexes $\alpha_1$ and $\alpha_2$ such that $T_\alpha(B_1)\subseteq\varepsilon V$ for each $\alpha\geq \alpha_1$ and $S_\alpha(B_2)\subseteq\varepsilon V$ for all $\alpha\geq\alpha_2$. Take an index $\alpha_0$ with $\alpha_0\geq\alpha_1$ and $\alpha_0\geq\alpha_2$, and order bounded set $B=B_1\cap B_2$. Then, for each $\alpha\geq\alpha_0$, we have 
$$
S_\alpha\big(T_\alpha(B)\big)\subseteq S_\alpha(\varepsilon V)=\varepsilon S_\alpha(V) \subseteq \varepsilon S_\alpha([-a,a]).
$$
By \cite[Thm.2.19(i)]{AB}, there is a positive scalar $\lambda$ such that $[-a,a]\subseteq \lambda V$. Then, for given positive scalar $\frac{\varepsilon}{\lambda}$, there exists an index $\alpha_3$ such that $T_\alpha(B)\subseteq\frac{\varepsilon}{\lambda} V$ and $S_\alpha(B)\subseteq\frac{\varepsilon}{\lambda} V$ for each $\alpha\geq \alpha_3$. Thus, we have 
$$
S_\alpha\big(T_\alpha(B)\big)\subseteq S_\alpha(\frac{\varepsilon}{\lambda}V) \subseteq \frac{\varepsilon}{\lambda} S_\alpha([-a,a])\subseteq \frac{\varepsilon}{\lambda} \lambda V=\varepsilon V
$$
for all $\alpha\geq \alpha_3$.
\end{proof}

In the following two works, we show the lattice operations are continuous with respect to the uniform convergence topology and the equicontinuous convergence topology, respectively.
\begin{thm}\label{lattice operations is cont.}
Let $X$ be a vector lattice and $(E,\tau)$ locally solid vector lattice with order complete property and an order bounded zero neighborhood. Then the lattice operations in $B_{ob}(X,E)$ are continuous with respect to the uniform convergence topology on order bounded sets.
\end{thm}

\begin{proof}
Assume $(T_\alpha)$ and $(S_\alpha)$ are two nets of $ob$-bounded operators which are uniform convergent to the linear operators $T$ and $S$ on order bounded sets, respectively. For each $x\in X_+$, by applying the Riesz–Kantorovich formula, we have
\begin{eqnarray}
\big(T\vee S\big)(x)=\sup\{Tu+Sv:u,v\geq0,u+v=x\}.
\end{eqnarray}

Fix an order bounded set $B$ and fix $x\in B_+$. Suppose $u$, $v$ are positive elements such that $x=u+v$, and so that $u,v\in B$. Also, for two subsets $A_1$, $A_2$ in a vector lattice, we have $\sup(A_1)-\sup(A_2)\leq \sup(A_1-A_2)$. Then we get
\begin{eqnarray*}
	\big(T_\alpha\vee S_\alpha\big)(x)-\big(T\vee S\big)(x)&=&\sup\{T_\alpha u+S_\alpha v:u,v\geq0,u+v=x\}\\&& -\sup\{Tu+Sv:u,v\geq0,u+v=x\}\\&\leq&\sup\{(T_\alpha-T) u+(S_\alpha-S) v:u,v\geq0,u+v=x\}.
\end{eqnarray*}
For given a zero neighborhood $V$ in $E$, pick a zero neighborhood $U$ with $U+U\subseteq V$. Thus, there are some indexes $\alpha_1$ and $\alpha_2$ such that $(T_\alpha-T)(B)\subseteq U$ for every $\alpha\geq\alpha_1$ and $(S_\alpha-S)(B)\subseteq U$ for each $\alpha\geq\alpha_2$. There exists another index $\alpha_0$ such that $\alpha_0\geq\alpha_1$ and $\alpha_0\geq\alpha_1$, so that $(T_\alpha-T)(B)\subseteq U$ and $(S_\alpha-S)\subseteq U$ for all $\alpha\geq\alpha_0$. Then, for each $\alpha\geq\alpha_0$, we have
$$
\big(T_\alpha\vee S_\alpha\big)(x)-\big(T\vee S\big)(x)\leq \big(T_\alpha-T\big)(x)+\big(S_\alpha-S\big)(x)\subseteq U+U\subseteq V.
$$
Therefore, we have $\big(T_\alpha\vee S_\alpha\big)(B)-\big(T\vee S\big)(B)\subseteq V$ for each $\alpha\geq\alpha_0$, and so we get the desired result.
\end{proof}

\begin{ques}
Does the Theorem \ref{lattice operations is cont.} hold without order bounded zero neighborhood?
\end{ques}

\begin{thm}\label{lattice operations is equi. cont.}
Let $X$ be a vector lattice and $(E,\tau)$ locally solid vector lattice with order complete property and an order bounded zero neighborhood. Then the lattice operations in $B_{ob}(X,E)$ are continuous with respect to the equicontinuous convergence topology with order bounded sets.
\end{thm}

\begin{proof}
Let $(T_\alpha)$ and $(S_\alpha)$ be nets of $ob$-bounded operators which are equicontinuous convergent to the linear operators $T$ and $S$ with order bounded sets, respectively. By the Riesz–Kantorovich formula, we have
$$
\big(T\vee S\big)(x)=\sup\{Tu+Sv:u,v\geq0,u+v=x\}.
$$
for ever $x\in X_+$. Fix a zero neighborhood $V$ in $E$. Take a zero neighborhood $U$ with $U+U\subseteq V$. There are order bounded sets $B_1$ and $B_2$ in $X$ such that, for every $\varepsilon>0$, there exist indexes $\alpha_1$ and $\alpha_2$ such that $T_\alpha(B_1)\subseteq\varepsilon U$ for each $\alpha\geq \alpha_1$ and $S_\alpha(B_2)\subseteq\varepsilon U$ for all $\alpha\geq\alpha_2$. Pick an index $\alpha_0$ with $\alpha_0\geq\alpha_1$ and $\alpha_0\geq\alpha_2$, and order bounded set $B=B_1\cap B_2$. Hence, $T_\alpha(B)\subseteq\varepsilon U$ and $S_\alpha(B)\subseteq\varepsilon U$ for all $\alpha\geq \alpha_0$. Fix $x\in B_+$, suppose $u$, $v$ are positive elements such that $x=u+v$, and so that $u,v\in B$. Then we get
\begin{eqnarray*}
\big(T_\alpha\vee S_\alpha\big)(x)-\big(T\vee S\big)(x)&=&\sup\{T_\alpha u+S_\alpha v:u,v\geq0,u+v=x\}\\&& -\sup\{Tu+Sv:u,v\geq0,u+v=x\}\\&\leq&\sup\{(T_\alpha-T) u+(S_\alpha-S) v:u,v\geq0,u+v=x\}.
\end{eqnarray*} 
Then, for each $\alpha\geq\alpha_0$, we have
$$
\big(T_\alpha\vee S_\alpha\big)(x)-\big(T\vee S\big)(x)\leq \big(T_\alpha-T\big)(x)+\big(S_\alpha-S\big)(x)\subseteq \varepsilon U+\varepsilon U\subseteq \varepsilon V.
$$
Therefore, we get the result.
\end{proof}

\begin{ques}
Does the Theorem \ref{lattice operations is equi. cont.} hold without order bounded zero neighborhood?
\end{ques}
 
We finish this paper with the following results which show that $B_{ob}(X,E)$ is topologically complete algebras with respect to the assigned topologies.
\begin{prop}\label{t is also ob-bounded}
Let $(T_\alpha)$ be a net of $ob$-bounded operators in $B_{ob}(X,E)$ which uniform convergent to the linear operators $T$ on order bounded sets. Then $T$ is $ob$-bounded.
\end{prop}

\begin{proof}
Let $B$ be an order bounded set. For given a zero neighborhood $V$ in $E$, there exists an index $\alpha_0$ such that  $\big(T-T_\alpha\big)(B)\subseteq V$ for all $\alpha\geq\alpha_0$. So, we have 
$$
T(B)\subseteq T_{\alpha_0}(B)+V.
$$ 
Since $T_{\alpha_0}$ is $ob$-bounded, $T_{\alpha_0}(B)$ is bounded in $E$. Also, since the sum of two bounded sets in a topological vector space is bounded, $T(B)$ is also bounded.
\end{proof}

\begin{ques}
Let $(T_\alpha)$ be a net of $ob$-bounded operators in $B_{ob}(X,E)$. Is $T$ $ob$-bounded operator whenever $(T_\alpha)$ equicontinuous convergent to the linear operators $T$ with order bounded sets?
\end{ques}

A net $(S_\alpha)$ on a locally solid vector lattice $(E,\tau)$ is said to {\em order converges to zero uniformly} on a bounded set $B$ if, for each $a\in E_+$, there is an $\alpha_0$ with $S_\alpha(B)\subseteq [-a,a]$ for each $\alpha\geq \alpha_0$; see \cite{Z}.
\begin{prop}\label{t is o}
Let $(T_\alpha)$ be a net of $ob$-bounded operators between locally solid vector lattices $(X,\acute{\tau})$ and $(E,\tau)$. If $(T_\alpha)$ order convergent
uniformly on zero neighborhood to the linear operator $T$ then $T$ is also $ob$-bounded.  
\end{prop}

\begin{proof}
Suppose $B$ is an order bounded set. So, $B$ is bounded in $X$; see \cite[Thm.2.19(i)]{AB}. Then, by assumption, for any $e\in E_+$ there exists an index $\alpha_0$ with $(T-T_\alpha)(B)\subseteq [-e,e]$ for each $\alpha\geq \alpha_0$. Thus, we have
$$
T(B)\subseteq T_{\alpha_0}(B)+[-e,e].
$$
Fix a $\tau$-neighborhood $V$, and consider another $\tau$-neighborhood $W$ with $W+W\subseteq V$. Since $T_{\alpha_0}$ is $ob$-bounded, there exists a positive scalar $\gamma_1$ such that $T_{\alpha_0}(B)\subseteq \gamma_1 W$. Also, by using \cite[Thm.2.19(i)]{AB}, there exists another positive scalar $\gamma_2$ such that $[-e,e]\subseteq \gamma_2 W$. Let's take $\gamma=\max\{\gamma_1,\gamma_2\}$, and so, by solidness of $W$, we have $\gamma_1 W\subseteq\gamma W$ and $\gamma_2 W\subseteq\gamma W$. Then we get
$$
T(B)\subseteq T_{\alpha_0}(B)+[-e,e]\subseteq\gamma_1 W+\gamma_2 W\subseteq\gamma W+\gamma W\subseteq \gamma V.
$$
Therefore we get the desired result.
\end{proof}

\begin{thm}
Let $(T_\alpha)$ be a net of $B_{ob}(X,E)$ which uniform convergent to the linear operator $T$ on order bounded sets. If $(E,\tau)$ is topologically complete vector lattice and every order bounded set in $E$ is absorbing then $B_{ob}(X,E)$ is complete with respect to the topology of uniform convergence on order bounded sets.
\end{thm}

\begin{proof}
Suppose $(E,\tau)$ is topological complete and $(T_\alpha)$ is a Cauchy net in $B_{ob}(X,E)$. Consider an order bounded set $B$ in $X$. Then, for each zero neighborhood $V$, there exists an index $\alpha_0$ such that 
$$
(T_\alpha-T_\beta)(B)\subseteq V
$$
for each $\alpha\geq\alpha_0$ and for each $\beta\geq\alpha_0$. For any $x\in X$, there is a positive real $\lambda$ such that $x\in \lambda B$ since $B$ is absorbing. Thus, $(T_\alpha-T_\beta)(x)\subseteq V$ for each $\alpha,\beta \geq\alpha_0$, so we conclude that $(T_\alpha(x))$ is a Cauchy net in $E$. Put $T(x)=\lim T_\alpha(x)$. Therefore, by Proposition \ref{t is also ob-bounded}, we get the desired result.
\end{proof}

\begin{thm}
Let $(T_\alpha)$ be a net of $ob$-bounded operators between locally solid vector lattices $(X,\acute{\tau})$ and $(E,\tau)$ which order convergent uniformly on zero neighborhood to the linear operator $T$. If $(E,\tau)$ is topologically complete vector lattice then $B_{ob}(X,E)$ is complete with respect to the topology of order convergent uniformly on zero neighborhood sets.
\end{thm}

\begin{proof}
Suppose $(E,\tau)$ is topological complete and $(T_\alpha)$ is a Cauchy net in $B_{ob}(X,E)$. Consider a bounded set $B$ in $X$. Thus, for each $a\in E_+$, there is an $\alpha_0$ with $$
(T_\alpha-T_\beta)(B)\subseteq [-a,a]
$$
for each $\alpha, \beta\geq \alpha_0$. For any $x\in X$, there is a positive real $\lambda$ such that $x\in \lambda B$ since $B$ is solid. Thus, $(T_\alpha-T_\beta)(x)\subseteq V$ for each $\alpha,\beta \geq\alpha_0$, so we conclude that $(T_\alpha(x))$ is a Cauchy net in $E$. Put $T(x)=\lim T_\alpha(x)$. Therefore, by Proposition \ref{t is o}, we get the desired result.
\end{proof}

\end{document}